\pdfoutput=1
\documentclass[a4paper]{article}
\usepackage[T2A]{fontenc}
\usepackage[cp1251]{inputenc}
\usepackage[russian,english]{babel}
\usepackage[tbtags]{amsmath}
\usepackage{amsfonts,amssymb,mathrsfs,amscd}
 \usepackage{msb-a}
 
\newlength\tindent
\renewcommand{\indent}{\hspace*{\tindent}}
\JournalName{}
\numberwithin{equation}{section}
\theoremstyle{plain}
\newtheorem{conjecture}{Conjecture}
\newtheorem{theorem}{Theorem}
\theoremstyle{plain}

\newtheorem{lemma}{Lemma}
\newtheorem{corollary}{Corollary}
\theoremstyle{definition}
\newtheorem{proof}{}

\newtheorem{remark}{Remark}
\renewcommand{\refname}{References}
\renewcommand{\abstract}{\begin{bf}{Abstract.}\end{bf}}
\begin{document}

\title{Intervals between numbers that are sums of two squares}
\author{Alexander Kalmynin}
\address{National Research University Higher School of Economics, Russian Federation, Math Department, International Laboratory of Mirror Symmetry and Automorphic Forms}
\email{alkalb1995cd@mail.ru}
\date{}
\udk{}
\maketitle
\markright{Intervals between sums of two squares}

\begin{fulltext}
\begin{abstract}
In this paper, we improve the moment estimates for the gaps between numbers that can be represented as a sum of two squares of integers. We consider certain sum of Bessel functions and prove the upper bound for its weighted mean value. This bound provides estimates for the $\gamma$-th moments of gaps for all $\gamma\leq 2$.
\end{abstract}
\section{Introduction}

Let $\mathcal S=\{s_1<s_2<\ldots<s_n<\ldots\}=\{1,2,4,5,8,\ldots\}\subset \mathbb N$ be the set of all natural numbers that are expressible as the sum of two squares of integers. The classical area in the research of the properties of this set is the study of gaps between consecutive elements, i.e. the quantity $s_{n+1}-s_n$ or, equivalently, the value distribution of the distance function of $\mathcal S$

$$R(x)=\min_{n \in \mathcal S} |x-n|$$

for $x\to +\infty$.

The following fact is well known (see \cite{Kar},\cite{BC}):

\begin{theorem}
The inequality

$$R(x)\ll x^{1/4}$$

holds.

\end{theorem}

\begin{proof}[Proof]
Let us note that if $f(x)=x-[\sqrt x]^2$, where $[y]$ is an integer part of the number $y$, then $0\leq f(x) \ll \sqrt x$ for $x\gg 1$  and $x-f(x)=[\sqrt x]^2$ is a square of integer. Therefore,

$$x^{1/4}\gg f(f(x))=x-(x-f(x)+f(x)-f(f(x)))\geq 0.$$

But $x-f(x)$ and $f(x)-f(f(x))$ are squares of integers, so for some integers $a$ and $b$ we have

$$|x-a^2-b^2|\ll x^{1/4},$$

which was to be proved.
\end{proof}

This estimate was probably known to L.\,Euler and, unfortunately, was not improved after that (it is still unknown if the identity $R(x)=o(x^{1/4})$ is true). Conjecturally, the correct order of growth of $R(x)$ is much smaller:

\begin{conjecture}

For any $\varepsilon>0$ the inequality

$$R(x)\ll_{\varepsilon} x^\varepsilon$$

is true.
\end{conjecture}

As for the large values of $R(x)$, by the work \cite{rich} of I.\,Richards, we have the following:

\begin{theorem}
For any $\varepsilon>0$ there exist infinitely many positive integers $x$ such that

$$R(x)>\left(\frac14-\varepsilon\right)\ln x.$$
\end{theorem}

In addition to the the upper and lower bounds, one can also consider the mean values of our function. The best result of this type was proved by C.\,Hooley\,\cite{hoo}:

\begin{theorem}
For any $0<\gamma<\frac53$ we have

$$\sum_{s_{n+1}\leq x} (s_{n+1}-s_n)^\gamma \ll x(\ln x)^{(\gamma-1)/2}.$$
\end{theorem}

Therefore, for almost all $n$ the inequality $s_{n+1}-s_n \ll \sqrt{\ln s_n}$ holds. More precisely:
\begin{corollary}

Let $g(x)$ be any function that tends to infinity. Then the number of $s_n \leq x$ with $s_{n+1}-s_n\gg g(x)\sqrt{\ln x}$ is $O(\frac{x}{g(x)^\gamma\sqrt{\ln x}})$ for any $\gamma<5/3$.
\end{corollary}
The main goal of the present work is to improve Hooley's theorem, i.e. to prove analogous estimate for the wider range of values of $\gamma$.

\begin{theorem}

For any $1<\gamma\leq 2$ the inequality

$$\sum_{s_{n+1}\leq x} (s_{n+1}-s_n)^\gamma \ll x(\ln x)^{\frac{3}{2}(\gamma-1)}\delta(x,\gamma)$$

is true, where

$$\delta(x,\gamma)=\begin{cases}
1,&\text{if $\gamma<2$;}\\
\ln x,&\text{if $\gamma=2$.}
\end{cases}$$
\end{theorem}
\begin{remark}
In 1986 V.\,A.\,Plaksin has published an article \cite{Pla} in which he presented the proof of the result similar to the Theorem 3, but for larger range of $\gamma$, namely $0<\gamma<2$. However, later his paper was shown to contain some mistakes (see \cite{Harm}) and attempts to fix the arguments were not successful (see \cite{Pla2}).
\end{remark}
\section{Proof of the main theorem}

In this section, we will prove Theorem 4 by reducing the original problem to the question about the distribution of values of certain sum of Bessel functions.

To do this, we need some lemmas.
Let us start with the transformation formula for the theta-function.
\begin{lemma}

Let  $M$ be a positive real number. For any real $x$ the equality

$$\vartheta_M(x):=\sum_{n \in \mathbb Z} e^{-\pi M(x+n)^2}=\frac{1}{\sqrt M} \sum_{m \in \mathbb Z} e^{2\pi i m x}e^{-\pi m^2/M}$$

is true.
\end{lemma}

\begin{proof}[Proof]

Consider the function $g(x)=e^{-\pi Mx^2}$ on the real line. It is easy to show that $g$ is a Schwartz function and

$$\vartheta_M(x)=\sum_{n \in \mathbb Z} g(x+n).$$

By the Poisson summation formula, for any $x \in \mathbb R$ we get

$$\sum_{n \in \mathbb Z} g(x+n)=\sum_{m \in \mathbb Z} e^{2\pi i m x}\widehat{g}(m),$$

where $\widehat{g}(\xi)=\int_{\mathbb R} g(x)e^{-2\pi i\xi x}dx$ is a Fourier transform of our function.

On the other hand, it is well known that

$$\int_{\mathbb R} e^{-\pi M x^2-2\pi i \xi x}dx=\frac{1}{\sqrt M}e^{-\pi \xi^2/M}.$$

Using this relation, we obtain the required result.

\end{proof}

With the help of Lemma 1 we will prove the following identity, which will be crucial for the subsequent considerations:
\begin{lemma}

Let $M$ and $N$ be some positive real numbers. Then

$$I(N,M):=\frac{1}{2\pi}\int_{-\pi}^{\pi} \vartheta_M(\sqrt N\sin \varphi)\vartheta_M(\sqrt N\cos \varphi)d\varphi=$$
$$=\frac{1}{M}\sum_{n \geq 0} r_2(n)J_0(2\pi \sqrt{Nn})e^{-\pi n/M}=:\frac{1}{M}S(N,M),$$

where $J_0(x)=\frac{1}{2\pi} \int\limits_{-\pi}^\pi e^{ix\cos\varphi}d\varphi$ is the Bessel function of the first kind of order zero and $r_2(n)$ is the number of pairs $(a,b)$ of integers such that $a^2+b^2=n$.

\end{lemma}

\begin{proof}[Proof]

Using lemma 1, we find

$$I(N,M)=\frac{1}{2\pi M} \int_{-\pi}^{\pi} \left(\sum_{a \in \mathbb Z} e^{2\pi i a\sqrt{N}\sin\varphi}e^{-\pi a^2/M}\right)\left(\sum_{b \in \mathbb Z} e^{2\pi i b \sqrt{N}\cos\varphi}e^{-\pi b^2/M}\right)d\varphi.$$

Both series are absolutely and uniformly convergent, so we can replace the product of their sums by the double sum and interchange summation and integration. Consequently,

$$I(N,M)=\frac{1}{2\pi M}\sum_{(a,b) \in \mathbb Z} e^{-(a^2+b^2)\pi/M}\int_{-\pi}^{\pi}e^{2\pi i\sqrt N(a \sin \varphi+b\cos\varphi)}d\varphi.$$ 

Let us now compute the inner integral for all integers $a$ and $b$. If $(a,b)=(0,0)$, then the integrand is equal to $1$ and so the integral equals $2\pi$. If, in contrast, $a^2+b^2 \neq 0$, then there exists some $\theta \in [-\pi,\pi]$ such that

$$\sin\theta=\frac{a}{\sqrt{a^2+b^2}}, \cos\theta=\frac{b}{\sqrt{a^2+b^2}}.$$

Therefore, $a\sin \varphi+b\cos\varphi=\sqrt{a^2+b^2}\cos(\varphi-\theta)$. Applying the change of variables $\varphi=\varphi_1+\theta$, we obtain

$$\int_{-\pi}^{\pi} e^{2\pi i \sqrt{N}(a\sin \varphi+b\cos\varphi)}d\varphi=\int_{-\pi-\theta}^{\pi-\theta} e^{2\pi i\sqrt{N(a^2+b^2)}\cos\varphi_1}d\varphi_1.$$

Since the integrand is periodic with the period $2\pi$, the last integral is equal to the integral of the same function over the interval $[-\pi,\pi]$. Thus, we finally find

$$\int_{-\pi}^{\pi} e^{2\pi i \sqrt{N}(a\sin \varphi+b\cos\varphi)}d\varphi=\int_{-\pi}^{\pi} e^{2\pi i\sqrt{N(a^2+b^2)}\cos\varphi_1}d\varphi_1=2\pi J_0(2\pi \sqrt{N(a^2+b^2)}).$$

Substituting the obtained result into the formula for $I(N,M)$, we deduce the identity

$$I(N,M)=\frac{1}{M}\sum_{(a,b) \in \mathbb Z} J_0(2\pi \sqrt{N(a^2+b^2)})e^{-\pi(a^2+b^2)/M}=\frac{1}{M}S(N,M),$$

which was to be proved.
\end{proof}
\begin{remark}
Computing the integral $I(N,M)$ without using Lemma 1 yields

$$I(N,M)=e^{-\pi NM}\sum_{n \geq 0} r_2(n)I_0(2\pi M\sqrt{Nn})e^{-\pi n M}.$$

Therefore, for all $N,M>0$ we have

$$Me^{-\pi NM}\sum_{n \geq 0} r_2(n)I_0(2\pi M\sqrt{Nn})e^{-\pi n M}=\sum_{n \geq 0} r_2(n)J_0(2\pi \sqrt{Nn})e^{-\pi n/M}.$$

Now, if for all $z, \tau \in \mathbb C$, $\mathrm{Im}\,\tau>0$ we have

$$f(\tau,z)=\sum_{n \geq 0} r_2(n)J_0(4\pi \sqrt{n}z)e^{\pi i n\tau}$$

then we get $S(N,M)=f\left(\frac{i}{M},\frac{\sqrt{N}}{2}\right)$ and by identity principle

$$f\left(-\frac{1}{\tau},\frac{z}{\tau}\right)=-i\tau\exp\left(\frac{4\pi i z^2}{\tau}\right)f(\tau,z).$$

So, the function $f(\tau,z)$ behaves like a Jacobi form. This phenomenon can be generalized to the Cohen-Kuznetsov series and provide some results on Rankin-Cohen brackets of modular forms (see \cite{Kuz} and \cite{Coh}).
\end{remark}

Now we want to show that if $R(N)$ is large, then the quantity $S(N,M)$ is rather small.
To prove this, we will need some more additional lemmas.
\begin{lemma}
Let $N>0$, $M>1$ and denote by $C(N) \subset \mathbb R^2$ the circle of radius $\sqrt{N}$ centered in the point $(0,0)$. Let $d(N)$ be the distance between the sets $C(N)$ and $\mathbb Z^2$, that is

$$d(N)=\inf_{\substack{ x \in C(N) \\ y \in \mathbb Z^2}} \rho(x,y),$$

where $\rho$ is a standard euclidean distance on the plane. Then the inequality

$$S(N,M)\leq Me^{-\pi Md(N)^2}+O(Me^{-\pi M/4})$$

holds.
\end{lemma}

\begin{proof}[Proof]
Let us notice first that for any $\varphi \in [-\pi,\pi]$ the inequality

$$||\sqrt{N}\cos\varphi||^2+||\sqrt{N}\sin\varphi||^2 \geq d(N)^2$$

holds. Indeed, for some integers $a$ and $b$ we have

$$||\sqrt{N}\cos\varphi||^2+||\sqrt{N}\sin\varphi||^2=(\sqrt{N}\cos\varphi-a)^2+(\sqrt{N}\sin\varphi-b)^2=$$
$$=\rho^2((\sqrt{N}\cos\varphi,\sqrt{N}\sin\varphi),(a,b))\geq d(N)^2,$$

because $(\sqrt{N}\cos\varphi,\sqrt{N}\sin\varphi) \in C(N)$ and $(a,b) \in \mathbb Z^2$. Then from the Lemma 2 we obtain

 $$S(N,M)=\frac{M}{2\pi}\int_{-\pi}^{\pi} \vartheta_M(\sqrt{N}\cos\varphi)\vartheta_M(\sqrt{N}\sin\varphi)d\varphi=$$
 $$=\frac{M}{2\pi}\int_{-\pi}^{\pi} \left(e^{-\pi M||\sqrt{N}\cos\varphi||^2}+O(e^{-\pi M/4})\right)\left(e^{-\pi M||\sqrt{N}\sin\varphi||^2}+O(e^{-\pi M/4})\right)d\varphi.$$
 
 Therefore, from the previous observation we get the desired inequality.
\end{proof}
It turns out that the distance between $C(N)$ and $\mathbb Z^2$ is closely related to the distance from $N$ to the nearest sum of two squares. More precisely, the following proposition holds:

\begin{lemma}
For any $N>0$ we have

$$d(N)\geq \frac{2R(N)}{5\sqrt{N}}$$
\end{lemma}

\begin{proof}[Proof]
Let $O=(0,0)$ and $A \in C(N), B \in \mathbb Z^2$ be the points such that

$$d(N)=\rho(A,B)$$

(points with this condition certainly exist because $C(N)$ is compact) Then the points $A, B$ and $O$ are collinear, because if this is not the case then the triangle $ABO$ is nondegenerate and thus

$$\rho(A,B)>|\rho(A,O)-\rho(B,O)|.$$

But then if we take $D$ to be the intersection point of the line $BO$ and the circle $C(N)$, which lies closer to $B$ than the second point, we get

$$\rho(D,B)=|\rho(D,O)-\rho(B,O)|=|\rho(A,O)-\rho(B,O)|<\rho(A,B),$$

which is a contradiction. Therefore,

$$d(N)=\rho(A,B)=|\rho(A,O)-\rho(B,O)|=|\sqrt{N}-\sqrt{a^2+b^2}|$$

for some integers $a$ and $b$. So we get

$$d(N)=\frac{|N-a^2-b^2|}{\sqrt{N}+\sqrt{a^2+b^2}}$$

and we clearly can assume that $a^2+b^2\leq 2N$, as otherwise we have

$$d(N)\geq (\sqrt{2}-1)\sqrt{N}>\frac{2}{5}\sqrt{N}\geq \frac{2R(N)}{5\sqrt{N}},$$

because $R(N)\leq |N-0|=N$. Now, if $a^2+b^2 \leq 2N$ then $\sqrt{N}+\sqrt{a^2+b^2} \leq (\sqrt{2}+1)\sqrt{N}<5\sqrt{N}/2$.  By the definition of $R(N)$ we also have

$$|N-a^2-b^2| \geq R(N)$$

and so

$$d(N) \geq \frac{R(N)}{5\sqrt{N}/2}=\frac{2R(N)}{5\sqrt{N}}.$$
\end{proof}

Combination of Lemma 3 and Lemma 4 gives us

\begin{lemma}
Let $H>1, N>3$ and $R(N) \geq H$. If $M \geq \frac{2N\ln N}{H^2}$ then 

$$S(N,M) \ll N^{-1/200}$$

\end{lemma}

\begin{proof}[Proof]
From the Lemma 3 we get

$$S(N,M) \ll Me^{-\pi Md(N)^2}+Me^{-\pi M/4}.$$

As $R(N)\ll N^{1/4}$, the inequality

$$M \gg \frac{N\ln N}{R(N)^2} \gg \sqrt{N}\ln N$$

holds and so

$$Me^{-\pi M/4} \ll N^{-1/200}.$$

Now, Lemma 4 gives us

$$d(N) \geq \frac{2R(N)}{5\sqrt{N}} \geq \frac{2H}{5\sqrt{N}}.$$

Consider the function

$$t(M)=Me^{-\pi d(N)^2 M}.$$

We have

$$t'(M)=(1-\pi d(N)^2 M)t(M)/M,$$

so the quantity $t(M)$ is decreasing when

$$M>\frac{1}{\pi d(N)^2}.$$

Furthermore,

$$M\geq \frac{2N\ln N}{H^2} \geq \frac{25N}{4\pi H^2}\geq\frac{1}{\pi d(N)^2}.$$

Therefore,

$$t(M) \leq t\left(\frac{2N\ln N}{H^2}\right) \ll Ne^{-2\pi Nd(N)^2\ln N/H^2}.$$

And now we have by the previous estimates

$$\frac{2\pi Nd(N)^2\ln N}{H^2} \geq \frac{8\pi}{25}\ln N>1.005\ln N,$$

 so
 
  $$t(M) \ll N^{-0.005}=N^{-1/200}.$$
  
 Consequently,
 
 $$S(N,M) \ll t(M)+Me^{-\pi M/4} \ll N^{-1/200},$$
 
 which was to be proved.
\end{proof}

So, if $R(N)$ is sufficiently large, then the quantity $S(N,M)$ is close to 0. But the function $J_0(x)$ is oscillating for $x\to \infty$ and so $S(x,M)-1$ is an infinite sum of certain oscillating functions. Therefore, because of possible cancellation it is reasonable to expect that $S(x,M)$ is close to $1$. More precisely, the following $L^2$-estimate holds:

\begin{lemma}
Let $N\geq 2$, $H \geq 40(\ln N)^{3/2}$ and $M=\frac{2N\ln N}{H^2} \geq 1$. Then the inequality

$$J(N,M)=\int_0^{N} (S(x,M)-1)^2dx\ll \sqrt{NM}\ln N$$

holds.
\end{lemma}

\begin{remark}
The constant $40$ here is not optimal, but this has no effect on subsequent considerations.
\end{remark}
To prove this lemma, we need three more propositions. First of them is very well known:

\begin{lemma}

For any $x \geq 1$ we have

$$\sum_{0<n< x} r_2^2(n) \ll x\ln x$$
\end{lemma}

We also need Weber's second exponential integral:

\begin{lemma}
For arbitrary $\alpha,\beta,\gamma>0$ the formula

$$\int_0^{+\infty} e^{-\alpha x}J_0(2\beta\sqrt{x})J_0(2\gamma\sqrt{x})dx=\frac{1}{\alpha}I_0\left(\frac{2\beta\gamma}{\alpha}\right)\exp\left(-\frac{\beta^2+\gamma^2}{\alpha}\right)$$

is true, where $I_0(x)=J_0(ix)$ is the modified Bessel function.

\end{lemma}

\begin{proof}[Proof]

See \cite{Wat}, p. 395, section 13.31.

\end{proof}

Furthermore, we will use the asymptotic formula for the modified Bessel function.

\begin{lemma}

For any positive real $x$ we have the following asymptotic formula:

$$I_0(x) \sim \frac{e^x}{\sqrt{2\pi x}}$$

\end{lemma}

\begin{proof}[Proof]

See \cite{Wat}, p. 203, section 7.23.
\end{proof}

\begin{proof}[Proof of the Lemma 6]
Let us define

$$J^*(N,M)=\int_0^{+\infty} (S(x,M)-1)^2e^{-\pi x/N}dx.$$

For any $x\leq N$ we have $e^{-\pi x/N}\geq e^{-\pi}$, so for all $N$ and $M$ we get

$$J(N,M)\leq e^{\pi} J^*(N,M)$$

and so it is enough to prove an analogous estimate for $J^*(N,M)$.

Now, by the definition of $S(x,M)$, we have

$$S(x,M)-1=\sum_{n\geq 1} r_2(n)J_0(2\pi \sqrt{nx})e^{-\pi n/M}$$

and the series converge absolutely and uniformly for $x\geq 0$. Therefore we obtain

$$J^*(N,M)=\int_0^{+\infty} (S(x,M)-1)^2e^{-\pi x/N}dx=$$
$$=\sum_{n,m\geq 1} r_2(n)r_2(m)e^{-\pi (n+m)/N} \int_0^{+\infty} J_0(2\pi \sqrt{mx})J_0(2\pi \sqrt{nx})e^{-\pi x/N}dx.$$

The integrals can be computed using the change of variables $x=\frac{N}{\pi}y$ and Lemma 8 as follows:

$$\int_0^{+\infty} J_0(2\pi \sqrt{mx})J_0(2\pi \sqrt{nx})e^{-\pi x/N}dx=\frac{N}{\pi}\int_0^{+\infty} J_0(2\sqrt{\pi nNy})J_0(2\sqrt{\pi m Ny})e^{-y}dy=$$
$$=\frac{N}{\pi}e^{-\pi N(n+m)}I_0(2\pi N\sqrt{nm}).$$

Therefore,
$$J^*(N,M)=\frac{N}{\pi}\sum_{n,m \geq 1}r_2(n)r_2(m)e^{-\pi N(n+m)-\pi (n+m)/m}I_0(2\pi N\sqrt{nm}).$$

We can split this sum into <<diagonal>> and <<non-diagonal>> summands as follows:

$$J^*(N,M)=\frac{N}{\pi}(S_0+S_1),$$

where

$$S_0=\sum_{n\geq 1} r_2(n)^2e^{-2\pi n/M}I_0(2\pi nN)e^{-2\pi nN}$$

and

$$S_1=\sum_{\substack{n,m \geq 1 \\ n\neq m}} r_2(n)r_2(m)e^{-\pi N(n+m)-\pi(n+m)/M}I_0(2\pi N\sqrt{nm}).$$

Let us estimate $S_1$ first. Note that all the summands are positive and for for any pair $(n,m)$ at least one of three pairs of inequalities is true:

$$1\leq n,m\leq 100M\ln N,$$

$$n \geq 1, m>100 M\ln N$$

or

$$m \geq 1, n>100 M\ln N.$$

Thus we deduce

$$S_1\leq \sum_{\substack{1\leq n,m \leq 100M\ln N \\ n\neq m}}r_2(n)r_2(m)e^{-\pi N(n+m)-\pi(n+m)/M}I_0(2\pi N\sqrt{nm})+$$
$$+2\sum_{n\geq 1,m>100M\ln N}r_2(n)r_2(m)e^{-\pi N(n+m)-\pi(n+m)/M}I_0(2\pi N\sqrt{nm})=S_2+2S_3.$$

To estimate the sum $S_3$, let us notice that Lemma 9 gives us

$$e^{-\pi N(n+m)}I_0(2\pi N\sqrt{nm})\ll e^{-\pi N(n+m)+2\pi N\sqrt{nm}}=e^{-\pi N(\sqrt{n}-\sqrt{m})^2} \leq 1.$$

And so we obtain the following inequality:

$$S_3 \ll \sum_{n\geq 1, m>100 M\ln N}r_2(n)r_2(m)e^{-\pi (n+m)/M}=$$
$$=\left(\sum_{n\geq 1} r_2(n)e^{-\pi n/M}\right)\left(\sum_{n>100M\ln N} r_2(n)e^{-\pi n/M}\right).$$

Now, for any $x>0$ we have

$$\sum_{n\leq x} r_2(n) \ll x,$$

so

$$\sum_{n\geq 1} r_2(n)e^{-\pi n/M}\leq \sum_{m=0}^{+\infty} e^{-\pi m}\sum_{n \leq (m+1)M} r_2(n) \ll \sum_{m=0}^{+\infty}(m+1)e^{-m}M \ll M$$

and

$$\sum_{n>100M\ln N} r_2(n)e^{-\pi n/M} \leq \sum_{m=100\ln N}^{+\infty} e^{-\pi m}\sum_{n \leq (m+1)M} r_2(n)\ll \sum_{m=100\ln N} e^{-\pi m}(m+1)M\ll$$
$$\ll M\ln Ne^{-100\pi \ln N} \ll N^{-313}$$

so we finally get

$$S_3 \ll MN^{-313} \ll N^{-312}.$$

Now we turn to $S_2$. For any integer pair $n,m$ with conditions $1\leq n,m \leq 100M\ln N$ we get by the Lemma 9

$$e^{-\pi N(n+m)}I_0(2\pi N\sqrt{nm})\ll e^{-\pi N(n+m-2\sqrt{nm})}.$$

Let us notice that as $n\neq m$ we have

$$n+m-2\sqrt{nm}=(\sqrt{n}-\sqrt{m})^2=\frac{(n-m)^2}{(\sqrt{n}+\sqrt{m})^2}\geq \frac{1}{400M\ln N}.$$

Consequently,

$$-\pi N(n+m)+2\pi N\sqrt{nm} \leq -\frac{\pi N}{400M\ln N}.$$

From this inequality and conditions $M=\frac{2N\ln N}{H^2}$ and $H \geq 40(\ln N)^{3/2}$ we obtain

$$-\pi N(n+m)+2\pi N\sqrt{nm} \leq -\frac{\pi H^2}{800\ln^2 N} \leq -\frac{1600\pi\ln^3 N}{800\ln^2 N} \leq -2\pi\ln N.$$

Therefore,

$$S_2 \ll \sum_{n,m \leq 100M\ln N}r_2(n)r_2(m)N^{-2\pi}=\left(\sum_{n \leq 100M\ln N}r_2(n)\right)^2N^{-2\pi}\ll N^{2-2\pi}\ll N^{-4}.$$

And we deduce for the <<non-diagonal>> summand the estimate

$$S_1 \ll N^{-4}+N^{-312} \ll N^{-4}.$$

It remains to prove that

$$NS_0 \ll \sqrt{NM}\ln N.$$

From the Lemma 9 we deduce

$$e^{-2\pi Nn}I_0(2\pi Nn) \ll \frac{1}{\sqrt{Nn}}.$$

Consequently,

$$S_0\ll \sum_{n\leq 100M\ln N} \frac{r_2(n)^2}{\sqrt{Nn}}e^{-\pi n/M}.$$

Let $K$ be the smallest integer with $2^{K+1}>100M\ln N$. Then we obviously have

$$S_0 \ll \sum_{k=0}^{K} S_0(k),$$

where

$$S_0(k)=\sum_{2^k\leq n<2^{k+1}} \frac{r_2(n)^2}{\sqrt{Nn}}e^{-\pi n/M}.$$

From the Lemma 7 we get

$$S_0(k) \leq \sum_{2^k\leq n<2^{k+1}} \frac{r_2(n)^2}{\sqrt{N}2^{k/2}}e^{-2^k\pi/M}\leq\frac{e^{-2^k\pi/M}}{\sqrt{N}2^{k/2}}\sum_{n<2^{k+1}}r_2(n)^2\ll \frac{(k+1)2^{k/2}e^{-2^k\pi/M}}{\sqrt{N}}.$$

Summing this over all $k\leq K$, we obtain

$$S_0 \ll \sqrt{\frac{M}{N}}\ln N.$$

Consequently,

$$J^*(N,M)=\frac{N}{\pi}(S_0+S_1)\ll \sqrt{NM}\ln N+N^{-3}\ll \sqrt{NM}\ln N$$

and

$$J(N,M)\leq e^\pi J^*(N,M) \ll \sqrt{NM}\ln N,$$

which concludes the proof.

\end{proof}

\begin{lemma}

Let $x,H>3$ and denote by $M(H,x)$ the set of all real $y \leq x$ with $R(y) \geq H$. Then the inequality

$$\mu(M(H,x)) \ll \frac{x(\ln x)^{3/2}}{H}$$

holds, where $\mu$ is the Lebesgue measure.

\end{lemma}

\begin{proof}[Proof]

As this bound is trivial for $H\leq 40(\ln x)^{3/2}$, we can assume that $H \geq 40(\ln x)^{3/2}$. Then from the Lemma 3 we get

$$|S(y,M)|\ll y^{-1/200}$$

for any $y\in M(H,x)$. Therefore for large enough $y \in M(H,x)$ we have

$$|S(y,M)-1| \gg 1.$$

Thus,

$$J(x,M)\geq \int_{M(x,H)} (S(y,M)-1)^2e^{-\pi y/N}dy\gg \mu(M(x,H))-O(1).$$

Consequently, by the Lemma 5,

$$\mu(M(x,H)) \ll \sqrt{xM}\ln x+O(1)=\frac{\sqrt{2}x(\ln x)^{3/2}}{H}+O(1)\ll \frac{x(\ln x)^{3/2}}{H},$$

which was to be proved.

\end{proof}

From this last lemma we deduce the Theorem 4.

\begin{proof}[Proof of Theorem 4]

Observe that $\sum\limits_{s_{n+1} \leq x} (s_{n+1}-s_n)^\gamma \ll \int\limits_0^x R(t)^{\gamma-1}dt$.

Indeed, for any positive integer $n$ we have

$$\int_{s_n}^{s_{n+1}}R(t)^{\gamma-1}dt=\int_{s_n}^{s_{n+1}} \min(t-s_n,s_{n+1}-t)^{\gamma-1}dt=\frac{(s_{n+1}-s_n)^\gamma}{2^{\gamma-1}\gamma}.$$

Summation over all $n$ with condition $s_{n+1} \leq x$ gives the desired result.

Let $k$ be some positive integer. Consider the set $B_k \subset [0,x]$ of all real $y$ with $2^k \leq R(y) \leq 2^{k+1}$.

Due to the Lemma 9, $\mu(B_k) \ll x(\ln x)^{3/2}2^{-k}$. Therefore,

$$\int_{B_k} R(t)^{\gamma-1}dt\leq 2^{(k+1)(\gamma-1)}\mu(B_k) \ll x(\ln x)^{3/2}2^{(\gamma-2)k}.$$

Let  $B$ be the set of $y \in [0,x]$ such that $R(y) \leq c(\ln x)^{3/2}$ with some sufficiently large constant $c$. Trivially, we have $B \subset [0,x]$ and so $\mu(B) \leq x$. Consequently,

$$\int_B R(t)^{\gamma-1}dt\ll \mu(B)(\ln x)^{\frac{3}{2}(\gamma-1)}\leq x(\ln x)^{\frac{3}{2}(\gamma-1)}.$$

Due to the fact that on the interval $[0,x]$ the inequality $R(x) \leq x$ holds, we have

$$[0,x]=B\cup\bigcup_{x \geq 2^k \geq c(\ln x)^{3/2}} B_k.$$

Therefore,

$$\int_0^x R(t)^{\gamma-1}dt=\int_B R(t)^{\gamma-1}dt+\sum_{x \geq 2^k \geq c(\ln x)^{3/2}} \int_{B_k} R(t)^{\gamma-1}dt.$$

Furthermore, we have

$$\int_0^x R(t)^{\gamma-1}dt \ll x(\ln x)^{\frac{3}{2}(\gamma-1)}+\sum_{x \geq 2^k \geq c(\ln x)^{3/2}} x(\ln x)^{3/2}2^{(\gamma-2)k} \ll x(\ln x)^{3/2(\gamma-1)}\delta(x,\gamma).$$

This concludes the proof of Theorem 4.
\end{proof}

\section{Conclusion}

In this work, we constructed certain sum of Bessel functions that is unusually small in the points that are far from numbers that are sums of two squares. The estimate for some mean value of this sum allowed us to prove the upper bound for the measure of the set of points with this property. However, we were not able to prove sharp enough bound for the sum $S(N,M)-1$. Nontrivial estimates for this quantity would allow us to improve the exponent in the inequality $R(x) \ll x^{1/4}$. One can show that our construction works not only for the sums of two squares, but also for the set of values of arbitrary positive definite quadratic form with integer coefficients. It would be also interesting to generalize this construction to the case of indefinite forms.

\section{Acknowledgments}

The author thanks Maxim Aleksandrovich Korolev for the useful comments and discussion.

The author is partially supported by Laboratory of Mirror Symmetry NRU HSE, RF Government grant, ag. \textnumero 14.641.31.0001, the Simons Foundation, the Moebius Contest Foundation for Young Scientists, the <<Young Russian Mathematics>> contest and the Program of the Presidium of the Russian Academy of Sciences \textnumero01 <<Fundamental Mathematics and its Applications>> under grant PRAS-18-01.
\end{fulltext}


\begin{thebibliography}{99}
\bibitem{BC} R.\,P.\,Bambah, Chowla,\,S., <<On numbers which can be expressed as a sum of two squares>>, Proc.\,Nat.\,Acad.\,Sci.\,India, 13, 1947, 101-103.
\bibitem{Coh} H.\,Cohen, <<Sums Involving the Values at Negative Integers of {L}-Functions of
Quadratic Characters>>, Mathematische Annalen 217 (1975): 271-285
\bibitem{Harm} G.\,Harman, <<Sums of two squares in short intervals>>, Proceedings of the London
Mathematical Society, 62, 1991, 225-241
\bibitem{hoo} C.\,Hooley, <<On the intervals between numbers that are sums of two squares I>>, Acta Math., 127, 1971, 279-297
\bibitem{Kar} A.\,A.\,Karatsuba, <<Euler and number theory>>, Proc. Steklov Inst. Math., 274, suppl. 1 (2011), 169-179
\bibitem{Kuz} N.\,V.\,Kuznetsov, <<A new class of identities for the Fourier coefficients of modular forms>> (in Russian),
Acta. Arith. 27 (1975) 505-519
\bibitem{Pla} V.\,A\,Plaksin, <<The distribution of numbers representable as a sum of two squares>>, Mathematics of the USSR-Izvestiya(1988),31(1):171
\bibitem{Pla2} V.\,A\,Plaksin, <<Letter to the editor: correction to the paper ''The distribution of numbers representable as a sum of two squares''>>, Russian Academy of Sciences. Izvestiya Mathematics(1993),41(1):187
\bibitem{rich} I.\,Richards, <<On the gaps between numbers which are sums of two squares>>, Advances in Math, 46, 1982, 1-2
\bibitem{Wat} G.\,N.\,Watson, <<A Treatise on the Theory of Bessel Functions>> (2nd.ed.), Cambridge University Press, 1966 
\end{thebibliography}
\end{document}